\documentclass[twoside,12pt]{article}

\usepackage[T1]{fontenc}
\usepackage[utf8]{inputenc}
\usepackage{lmodern}
\usepackage{microtype}
\usepackage{amsmath,amssymb,amsthm,mathtools}
\usepackage{xcolor}
\usepackage[a4paper,top=2cm,bottom=2cm,left=2cm,right=2cm]{geometry}
\usepackage{enumitem}
\usepackage[colorlinks=true,linkcolor=blue,citecolor=blue,urlcolor=blue]{hyperref}
\usepackage[nameinlink,noabbrev]{cleveref}

\hypersetup{
  pdftitle={New upper bound for multicolor Ramsey numbers},
  pdfauthor={Gang Yang and Yaping Mao},
  pdfsubject={Upper bounds for multicolor diagonal Ramsey numbers},
  pdfkeywords={Ramsey numbers, book method, higher-order correlation, root filters, multinomial coefficients}
}

\allowdisplaybreaks
\setlist[enumerate]{itemsep=1pt,topsep=4pt,parsep=0pt,partopsep=0pt}

\newtheorem{theorem}{Theorem}[section]
\newtheorem{lemma}[theorem]{Lemma}
\newtheorem{corollary}[theorem]{Corollary}

\crefname{theorem}{theorem}{theorems}
\Crefname{theorem}{Theorem}{Theorems}
\crefname{lemma}{lemma}{lemmas}
\Crefname{lemma}{Lemma}{Lemmas}
\crefname{corollary}{corollary}{corollaries}
\Crefname{corollary}{Corollary}{Corollaries}

\newcommand{\norm}[1]{\left\lVert #1\right\rVert}
\newcommand{\E}{\mathbb E}
\newcommand{\Pp}{\mathbb P}
\newcommand{\R}{\mathbb R}
\newcommand{\C}{\mathbb C}
\newcommand{\1}{\mathbf 1}
\newcommand{\ip}[2]{\left\langle #1,#2\right\rangle}
\newcommand{\tensor}{\otimes}
\newcommand{\dd}{\,\mathrm d}

\newcommand{\M}{\mathcal M}
\newcommand{\whM}{\widehat{\mathcal M}}
\newcommand{\bb}{\mathbf b}
\newcommand{\ee}{\mathbf e}
\begin{document}

\title{New upper bound for multicolor Ramsey numbers}
\author{Gang Yang\footnote{Graduate School of Environment and Information Sciences, Yokohama National University, 79-2 Tokiwadai, Hodogaya-ku, Yokohama 240-8501, Japan. {\tt gangyang98@outlook.com}},\quad Yaping Mao\footnote{Corresponding author. School of Mathematics and Statistics, Qinghai Normal University, Xining 810008, China. {\tt maoyaping@outlook.com}.}
}
\date{}
\maketitle

\begin{abstract}
For integers $r,k\ge2$, the diagonal Ramsey number $R_r(k)$ is the least
integer $n$ such that every $r$-coloring of the edges of $K_n$ contains a
monochromatic $K_k$. We prove that there is an absolute constant $c>0$
such that, for every $r,k\ge3$,
\[
 R_r(k)\le\frac{2}{r!}\frac{(r(k-1))!}{((k-1)!)^r}
 \exp\!\left(-\frac{ck}{\log^2(2r)}\right).
\]


\end{abstract}

\section{Introduction}\label{sec:introduction}

For integers $r,k\ge2$, let $R_r(k)$ denote the least integer $n$ such that
every coloring of the edges of the complete graph $K_n$ with colors from
$[r]=\{1,\ldots,r\}$ contains a monochromatic copy of $K_k$. Colors may
be unused. Ramsey's theorem guarantees that this number is finite
\cite{Ramsey1930}. The classical Erd\H{o}s--Szekeres recursion
\cite{ErdosSzekeres1935} gives
\[
 R_r(k)\le M_r(k):=\frac{(r(k-1))!}{((k-1)!)^r}\le r^{r(k-1)}.
\]
Throughout the paper all graphs are finite and simple, all logarithms
are natural, and
\[
 s_r=\log(2r).
\]
Constants described as absolute are independent of all integer parameters.
We use $f=O(g)$ and $f=\Omega(g)$ for upper and lower bounds by positive
absolute multiples, respectively; subscripts record any permitted
dependence. We write $(x)_+=\max\{x,0\}$ and $\1_A$ for an indicator.

For two colors, Thomason, Conlon, and Sah obtained successively stronger
subexponential improvements to the classical bound
\cite{Thomason1988,Conlon2009,Sah2023}. Campos, Griffiths, Morris, and
Sahasrabudhe proved an exponential improvement
\cite{CamposEtAl2026}. Gupta, Ndiaye, Norin, and Wei subsequently obtained
$R_2(k)\le3.8^{k+o(k)}$ \cite{GuptaEtAl2024}. Erd\H{o}s's probabilistic
construction gives an exponential lower bound \cite{Erdos1947}.
For background on these questions and the book method, see
\cite{ConlonFoxSudakov2015,Wigderson2025}.

Balister, Bollob\'as, Campos, Griffiths, Hurley, Morris, Sahasrabudhe,
and Tiba extended the exponential improvement to every fixed number of
colors \cite{BalisterEtAl2026}. Their explicit estimate is
\[
 R_r(k)\le r^{rk}\exp(-2^{-160}r^{-12}k)
 \qquad(k\ge2^{200}r^{20}).
\]
Narang and Tang used robust OR polynomials to improve the dependence on
$r$ in the underlying correlation estimate \cite{NarangTang2026}, obtaining
\begin{equation*}
 R_r(k)\le r^{rk}
 \exp\!\left(-\frac{ck}{r^9\log^6(2r)}\right)
 \qquad(k\ge Cr^{14}\log^{12}(2r))
\end{equation*}
for absolute $c,C>0$.

We prove the following uniform bound.

\begin{theorem}\label{thm:main}
There is an absolute constant $c>0$ such that, for all integers $r,k\ge3$,
\begin{equation*}
 R_r(k)\le M_r(k)\min\left\{
 \frac{3}{(2r-1)!!},\,
 \frac{2}{r!}\exp\!\left(-\frac{ck}{s_r^2}\right)
 \right\},
\end{equation*}
where $(2r-1)!!=1\cdot3\cdot5\cdots(2r-1)$.
\end{theorem}

For sufficiently large $k$, the exponential bound also holds when $r=2$.

\begin{theorem}\label{thm:large}
There are absolute constants $c_0,C_0>0$ such that, for all integers
$r\ge2$ and $k\ge C_0rs_r^2$,
\begin{equation*}
 R_r(k)\le\frac{2}{r!}M_r(k)
 \exp\!\left(-\frac{c_0k}{s_r^2}\right).
\end{equation*}
\end{theorem}

\begin{corollary}\label{cor:all-parameters}
There is an absolute constant $c>0$ such that, for all integers $r,k\ge2$,
\begin{equation*}
 R_r(k)\le\min\left\{M_r(k),\,
 r^{rk}\exp\!\left(-\frac{ck}{s_r^2}\right)\right\}.
\end{equation*}
\end{corollary}

The exclusions in Theorem~\ref{thm:main} are necessary for its displayed
exponential expression: $(2/r!)M_r(2)=R_r(2)=2$, while
$M_2(3)=R_2(3)=6$. The coefficient $2/r!$ comes from the subtractive
term in the classical recursion. Section~\ref{sec:prelim} gives the
corresponding elementary argument and the double-factorial refinement.

The numerical constants are not optimized. The specialized two-color
estimates remain stronger when $r=2$.

Let $\sigma_i$ be maps from a finite probability space into real Hilbert
spaces. For independent samples $U,U'$, write
$Z_i=\langle\sigma_i(U),\sigma_i(U')\rangle$.
For every integer $d\ge3$, there exist a coordinate $i$ and a threshold
$\lambda\ge-1$ such that
\[
 \Pp(Z_i\ge\lambda,\ Z_j\ge-1\text{ for }j\ne i)
 \ge\beta_r\exp[-C_{r,d}(\lambda+1)^{1/d}],
\]
where $\log(1/\beta_r)=O(r)$ and
$C_{r,d}=O(rd^2\log(2rd))$. This follows from the nonnegativity of mixed
Gram moments and a roots-of-unity filter for
$E_d(z)=\sum_{n\ge0}z^n/(dn)!$.

The book construction maintains pairwise disjoint color-$i$ cliques
$C_i$ and a common reservoir $W$ joined to $C_i$ in color $i$. With
$K=k-1$ and remaining capacities $b_i=K-|C_i|$, avoidance of $K_k$ gives
\[
 \log\frac{|W|}{(2/r!)\,\mathcal M(\bb)}<0,
 \qquad
 \mathcal M(\bb)=\frac{(\sum_i b_i)!}{\prod_i b_i!}.
\]
A vertex of sufficiently large color degree increases this potential.
Otherwise, a weighted density estimate gives one large vertex class
with a small total deficit. We assign spine targets to colors that can
support a book and clique targets without pages to the other colors.
The same estimate bounds the sum of these targets. For a completed book,
we retain its spine and whole page; exact factorial cancellation gives a
positive potential gain for each vertex added to an outer clique. Once
one quarter of the total capacity has been used, the accumulated gain
contradicts the displayed inequality.

After the preliminaries in Section~\ref{sec:prelim}, we prove the
correlation inequality in Section~\ref{sec:analytic} and the book lemma
in Section~\ref{sec:book}. Section~\ref{sec:capacity} contains the proofs
of the main theorems.

\section{Preliminaries}\label{sec:prelim}

Given an $r$-edge-colored complete graph, a color $i\in[r]$, and a vertex
$v$, let
$
 N_i(v)=\{w:vw\text{ has color }i\}
$
be the color-$i$ neighborhood of $v$. If $X$ and $Y$ are nonempty vertex
sets, define the minimum color-$i$ density from $X$ to $Y$ by
\[
 p_i(X,Y)=\min_{x\in X}\frac{|N_i(x)\cap Y|}{|Y|}.
\]
The sets $X$ and $Y$ are not required to be disjoint.

A \emph{color-$i$ book} is an ordered pair $(T,P)$ of disjoint vertex sets
such that $T$ is a color-$i$ clique and every edge between $T$ and $P$ has
color $i$. We call $T$ the \emph{spine} and $P$ the \emph{page set}.
This terminology is consistent with that used in the recent diagonal Ramsey work
\cite{CamposEtAl2026,Wigderson2025} and in its multicolor development
\cite{BalisterEtAl2026}.

For positive integers $k_1,\ldots,k_r$, let $R(k_1,\ldots,k_r)$ denote the
least $n$ with the following property: every $r$-edge-coloring of $K_n$
contains a color-$i$ copy of $K_{k_i}$ for some $i\in[r]$. For nonnegative
integers $b_1,\ldots,b_r$ whose sum is $B$, write
\[
 \binom{B}{b_1,\ldots,b_r}=\frac{B!}{b_1!\cdots b_r!}.
\]
For $\bb=(b_1,\ldots,b_r)$ with positive integer entries, set $\M(\bb)=(\sum_i b_i)!/\prod_i b_i!$ and $\whM(\bb)=2\M(\bb)/r!$. Both formulas also apply to nonnegative entries. Write $\ee_i$ for the $i$th coordinate vector. For $a_i\ge2$, the standard vertex-neighborhood argument gives
\begin{equation}\label{eq:rec}
 R(a_1,\ldots,a_r)
 \le \sum_{i=1}^rR(a_1,\ldots,a_i-1,\ldots,a_r)-(r-2).
\end{equation}
Indeed, if $n$ is the expression on the right and no target occurs, every
color-$i$ degree is at most
$R(a_1,\ldots,a_i-1,\ldots,a_r)-1$. Their sum is at most $n-2$,
contradicting the degree $n-1$ of a vertex.

\begin{lemma}\label{lem:factorial}
For every $r\ge2$ and every positive integer vector $\bb$,
\begin{equation}\label{eq:factorial}
 R(b_1+1,\ldots,b_r+1)\le\whM(\bb).
\end{equation}
The coefficient $2/r!$ is the smallest coefficient valid uniformly over
all such vectors.
\end{lemma}
\begin{proof}
Induct on $B$. When $b_i\ge2$, the induction hypothesis bounds the
$i$th term in \eqref{eq:rec} by $\whM(\bb-\ee_i)$. When $b_i=1$,
that Ramsey number is $1$. All other entries are positive, so
$\M(\bb-\ee_i)\ge(r-1)!$ and
\[
 1\le\whM(\bb-\ee_i)+1-\frac2r.
\]
The multinomial identity
$\sum_i\M(\bb-\ee_i)=\M(\bb)$ and \eqref{eq:rec} now give
\[
 R(b_1+1,\ldots,b_r+1)
 \le\whM(\bb)+m(1-2/r)-(r-2)\le\whM(\bb),
\]
where $m=|\{i:b_i=1\}|\le r$. At the initial vector
$\bb=(1,\ldots,1)$, both sides of \eqref{eq:factorial} equal $2$.
This also proves the optimality of the uniform coefficient.
\end{proof}

\begin{lemma}\label{lem:slack}
Set $d_2=1$ and $d_r=d_{r-1}/(2r-1)$ for $r\ge3$, so that
$d_r=3/(2r-1)!!$. If $b_i\ge2$ for all $i$, then
\begin{equation}\label{eq:double}
 R(b_1+1,\ldots,b_r+1)\le d_r\M(\bb).
\end{equation}
Consequently, with $\alpha=\frac13\log(5/3)>0$, for $r,k\ge3$,
\begin{equation}\label{eq:slack}
 R_r(k)\le \frac{2}{r!}\frac{(r(k-1))!}{((k-1)!)^r}e^{-\alpha r}.
\end{equation}
\end{lemma}
\begin{proof}
For $r=2$, \eqref{eq:double} is Lemma~\ref{lem:factorial}.
Induct first on $r$, and then on $B$. In \eqref{eq:rec}, a term with
$b_i\ge3$ is at most $d_r\M(\bb-\ee_i)$ by induction on $B$.
If $b_i=2$, its target in color $i$ is $2$. A coloring avoiding that
target has no edge of color $i$, and hence this term is exactly the
$(r-1)$-color Ramsey number on the remaining targets. It is at most
\[
 d_{r-1}\frac{(B-2)!}{\prod_{j\ne i}b_j!}
 =\frac{d_{r-1}}{B-1}\M(\bb-\ee_i)
 \le d_r\M(\bb-\ee_i),
\]
since $B\ge2r$. Sum these bounds in \eqref{eq:rec} and discard its
nonpositive subtraction. This proves \eqref{eq:double}.

Put $h_r=r!d_r/2$. Then $h_2=1$ and
$h_r/h_{r-1}=r/(2r-1)\le3/5$ for $r\ge3$. Thus
$h_r\le(3/5)^{r-2}\le e^{-\alpha r}$, because $r-2\ge r/3$.
Apply \eqref{eq:double} with $b_i=k-1$.
\end{proof}

Writing $B=k_1+\cdots+k_r$, Lemma~\ref{lem:factorial} gives the following standard bound when every $k_i\ge2$. If some $k_i=1$, the Ramsey number is $1$, and the same bound holds because its multinomial coefficient is at least $1$. Thus
\begin{equation}\label{eq:off-diagonal-standard}
 R(k_1,\ldots,k_r)
 \le \binom{B-r}{k_1-1,\ldots,k_r-1}.
\end{equation}
For a real Hilbert space $\mathcal H$, write $\ip{x}{y}_{\mathcal H}$
and $\norm{x}_{\mathcal H}$ for its inner product and norm, omitting
subscripts when the space is clear. Tensor products are over $\R$ and
have the inner product
\[
 \ip{\bigotimes_{j=1}^m x_j}{\bigotimes_{j=1}^m y_j}
 =\prod_{j=1}^m\ip{x_j}{y_j}.
\]
For a nonnegative integer $a$, let $\mathcal H^{\tensor a}$ and
$x^{\tensor a}$ denote the corresponding tensor powers, with
$\mathcal H^{\tensor0}=\R$ and $x^{\tensor0}=1$. Empty products equal $1$.

\begin{lemma}\label{lem:positive-moments}
Let $r\ge1$ be an integer. Let $U$ and $U'$ be independent copies of a
random variable taking values in a finite set $\mathcal X$. For each
$i\in[r]$, let $\mathcal H_i$ be a real Hilbert space and let
$\sigma_i:\mathcal X\to\mathcal H_i$ be a map. Then, for all nonnegative
integers $a_1,\ldots,a_r$,
\begin{equation*}
 \E\left(\prod_{i=1}^r
 \ip{\sigma_i(U)}{\sigma_i(U')}_{\mathcal H_i}^{a_i}\right)
 \ge0.
\end{equation*}
\end{lemma}

\begin{proof}
Let $\mu(x)=\Pp(U=x)$ and set
\[
 Z(x)=\bigotimes_{i=1}^r\sigma_i(x)^{\tensor a_i}
 \in\mathcal K:=\bigotimes_{i=1}^r\mathcal H_i^{\tensor a_i}.
\]
The tensor-product identity gives
$\ip{Z(x)}{Z(y)}_{\mathcal K}
=\prod_i\ip{\sigma_i(x)}{\sigma_i(y)}_{\mathcal H_i}^{a_i}$.
Since $U,U'$ are independent and $\mathcal X$ is finite,
\begin{align*}
 \E\left(\prod_{i=1}^r
 \ip{\sigma_i(U)}{\sigma_i(U')}_{\mathcal H_i}^{a_i}\right)
 &=\sum_{x,y\in\mathcal X}\mu(x)\mu(y)\ip{Z(x)}{Z(y)}_{\mathcal K}\\
 &=\norm{\sum_{x\in\mathcal X}\mu(x)Z(x)}_{\mathcal K}^{2}\ge0.
\end{align*}
\end{proof}

\section{A correlation inequality}\label{sec:analytic}

Write $\mathrm i=\sqrt{-1}$ and $\operatorname{Re}(z)$ for the real part
of $z\in\C$. Fix an integer $d\ge3$, and define
\[
 E_d(z)=\sum_{n=0}^{\infty}\frac{z^n}{(dn)!}.
\]
This series defines an entire function with nonnegative Taylor coefficients. Put
\[
 \theta_d=\frac{\pi}{d},
 \qquad
 \eta_d=1-\cos(2\theta_d),
 \qquad
 \gamma_d=1-\cos\theta_d.
\]

\begin{lemma}\label{lem:root-unity}
Let $d\ge3$ be an integer, and let
\[
 \omega_d=\exp\!\left(\frac{2\pi\mathrm i}{d}\right).
\]
Then, for every $u\ge0$,
\begin{align}
 E_d(u^d)
 &=\frac1d\sum_{j=0}^{d-1}\exp(u\omega_d^j),
 \label{eq:positive-root-filter}\\
 E_d(-u^d)
 &=\frac1d\sum_{j=0}^{d-1}
 \exp\!\left(u\exp\!\left(\frac{(2j+1)\pi\mathrm i}{d}\right)\right).
 \label{eq:negative-root-filter}
\end{align}
Consequently, for every $u\ge0$,
\begin{equation}\label{eq:negative-axis-bound}
 |E_d(-u^d)|\le e^{u\cos\theta_d}
\end{equation}
and
\begin{equation}\label{eq:positive-axis-upper}
 E_d(u^d)\le e^u.
\end{equation}
Moreover, for every $u\ge0$ satisfying
\begin{equation}\label{eq:positive-axis-condition}
 u\eta_d\ge\log\bigl(2(d-1)\bigr),
\end{equation}
one has
\begin{equation}\label{eq:positive-axis-lower}
 E_d(u^d)\ge\frac{e^u}{2d}.
\end{equation}
\end{lemma}

\begin{proof}
For every nonnegative integer $m$, the standard roots-of-unity identity gives
\begin{equation}\label{eq:root-orthogonality}
 \frac1d\sum_{j=0}^{d-1}\omega_d^{jm}
 =
 \begin{cases}
  1,&d\mid m,\\
  0,&d\nmid m.
 \end{cases}
\end{equation}
This follows by summing a geometric progression. Expanding the
exponential and using absolute convergence, we obtain
\begin{align*}
 \frac1d\sum_{j=0}^{d-1}\exp(u\omega_d^j)
=\sum_{m=0}^{\infty}\frac{u^m}{m!}
 \left(\frac1d\sum_{j=0}^{d-1}\omega_d^{jm}\right)=\sum_{n=0}^{\infty}\frac{u^{dn}}{(dn)!}=E_d(u^d).
\end{align*}
This proves \eqref{eq:positive-root-filter}.

For $j=0,\ldots,d-1$, put
\[
 \zeta_j
 =\exp\!\left(\frac{(2j+1)\pi\mathrm i}{d}\right)
 =\exp\!\left(\frac{\pi\mathrm i}{d}\right)\omega_d^j.
\]
The numbers $\zeta_0,\ldots,\zeta_{d-1}$ are precisely the $d$th roots of
$-1$. For every nonnegative integer $m$,
\[
 \frac1d\sum_{j=0}^{d-1}\zeta_j^m
 =\exp\!\left(\frac{\pi\mathrm i m}{d}\right)
 \left(\frac1d\sum_{j=0}^{d-1}\omega_d^{jm}\right).
\]
By \eqref{eq:root-orthogonality}, this expression vanishes unless $m=dn$
for some nonnegative integer $n$. If $m=dn$, it equals
$\exp(\pi\mathrm i n)=(-1)^n$. Therefore
\begin{align*}
 \frac1d\sum_{j=0}^{d-1}\exp(u\zeta_j)
=\sum_{m=0}^{\infty}\frac{u^m}{m!}
 \left(\frac1d\sum_{j=0}^{d-1}\zeta_j^m\right)=\sum_{n=0}^{\infty}\frac{(-1)^nu^{dn}}{(dn)!}=E_d(-u^d),
\end{align*}
which proves \eqref{eq:negative-root-filter}.

For every $j\in\{0,\ldots,d-1\}$,
\[
 \operatorname{Re}(\zeta_j)
 =\cos\!\left(\frac{(2j+1)\pi}{d}\right)
 \le\cos\!\left(\frac\pi d\right)=\cos\theta_d.
\]
Hence,
$
 |\exp(u\zeta_j)|
=\exp\!\left(u\operatorname{Re}(\zeta_j)\right)
 \le e^{u\cos\theta_d}.
$
Applying the triangle inequality to \eqref{eq:negative-root-filter} proves
\eqref{eq:negative-axis-bound}.

It remains to prove the positive-axis estimates. Since all terms in the
defining series of $E_d(u^d)$ are nonnegative,
\[
 E_d(u^d)
 =\sum_{n=0}^{\infty}\frac{u^{dn}}{(dn)!}
 \le\sum_{m=0}^{\infty}\frac{u^m}{m!}=e^u,
\]
which proves \eqref{eq:positive-axis-upper}. For the lower bound, in
\eqref{eq:positive-root-filter} the term corresponding to $j=0$ is $e^u$.
For every $j\in\{1,\ldots,d-1\}$,
\[
 \operatorname{Re}(\omega_d^j)
 =\cos\!\left(\frac{2\pi j}{d}\right)
 \le\cos\!\left(\frac{2\pi}{d}\right)=\cos(2\theta_d).
\]
Since $E_d(u^d)$ is real, taking real parts in
\eqref{eq:positive-root-filter} and using the triangle inequality gives
\[
 dE_d(u^d)
 \ge e^u-(d-1)e^{u\cos(2\theta_d)}
 =e^u\left(1-(d-1)e^{-u\eta_d}\right).
\]
Condition \eqref{eq:positive-axis-condition} implies
$(d-1)e^{-u\eta_d}\le1/2$, and hence
$E_d(u^d)\ge e^u/(2d)$. This proves
\eqref{eq:positive-axis-lower}.
\end{proof}

Define
\[
 u_{r,d}=\max\left\{
 1,
 \frac{\log(2(d-1))}{\eta_d},
 \frac{\log(24d^2r)}{2\gamma_d}
 \right\},
 \qquad
 L_{r,d}=u_{r,d}^d,
\]
and
$a_{r,d}=\exp(-2u_{r,d}\cos\theta_d)$.
The inequalities
\[
 1-\cos x\ge\frac{2x^2}{\pi^2}
 \qquad(0\le x\le\pi)
\]
and
\[
 1-\cos(2x)\ge\frac{8x^2}{\pi^2}
 \qquad(0\le x\le\pi/2)
\]
give
\[
 \gamma_d\ge\frac{2}{d^2},
 \qquad
 \eta_d\ge\frac{8}{d^2}.
\]
Consequently, there is an absolute constant $C_0>0$ such that
\begin{equation}\label{eq:urd-bound}
 u_{r,d}\le C_0d^2\log(2rd).
\end{equation}

Define
\[
 H_{r,d}(z)=1+a_{r,d}E_d(z)^2,
 \qquad
 G_{r,d}(z)=\frac{H_{r,d}(z)-H_{r,d}(-z)}2.
\]
The Taylor coefficients of $H_{r,d}$ are nonnegative, and $G_{r,d}$ is the
odd part of $H_{r,d}$, so its Taylor coefficients are nonnegative as well.
Also, $H_{r,d}(x)\ge1$ for every real $x$.

\begin{lemma}\label{lem:ratio-estimates}
Let $r\ge2$ and $d\ge3$ be integers. Then, for every $x\ge0$,
\[
 0\le\frac{G_{r,d}(x)}{H_{r,d}(x)}\le\frac12.
\]
For every $x<0$,
\[
 \frac{G_{r,d}(x)}{H_{r,d}(x)}\le0.
\]
Moreover, if $y\ge L_{r,d}$, then
\[
 \frac{G_{r,d}(-y)}{H_{r,d}(-y)}\le-r.
\]
\end{lemma}

\begin{proof}
Since $E_d(x)$ is real for every real $x$,
\[
 H_{r,d}(x)=1+a_{r,d}E_d(x)^2\ge1.
\]
Let $x\ge0$. The nonnegativity of the Taylor coefficients of $E_d$ gives
$|E_d(-x)|\le E_d(x)$.
Therefore, $1\le H_{r,d}(-x)\le H_{r,d}(x)$,
and hence
\[
 \frac{G_{r,d}(x)}{H_{r,d}(x)}
 =\frac12\left(1-\frac{H_{r,d}(-x)}{H_{r,d}(x)}\right)
 \in\left[0,\frac12\right].
\]

If $x<0$, write $x=-z$ with $z>0$. Since $G_{r,d}$ is odd and
$G_{r,d}(z)\ge0$, we have $G_{r,d}(x)\le0$. Because $H_{r,d}(x)>0$, the
second assertion follows.

Now let $y\ge L_{r,d}$. Write $y=u^d$, where $u\ge u_{r,d}$. By the
definition of $u_{r,d}$ and \Cref{lem:root-unity},
\[
 E_d(y)\ge\frac{e^u}{2d}.
\]
Thus
\begin{equation}\label{eq:ratio-H-positive}
 H_{r,d}(y)
 \ge a_{r,d}E_d(y)^2
 \ge a_{r,d}\frac{e^{2u}}{4d^2}.
\end{equation}
On the other hand, \eqref{eq:negative-axis-bound} gives
\[
 H_{r,d}(-y)
 \le1+a_{r,d}e^{2u\cos\theta_d}.
\]
Since $u\ge u_{r,d}$,
$a_{r,d}e^{2u\cos\theta_d}
 =e^{2(u-u_{r,d})\cos\theta_d}\ge1$,
and therefore
\begin{equation}\label{eq:ratio-H-negative}
 H_{r,d}(-y)\le2a_{r,d}e^{2u\cos\theta_d}.
\end{equation}
Combining \eqref{eq:ratio-H-positive} and
\eqref{eq:ratio-H-negative}, we obtain
\[
 \frac{H_{r,d}(y)}{H_{r,d}(-y)}
 \ge\frac{e^{2u(1-\cos\theta_d)}}{8d^2}
 =\frac{e^{2\gamma_du}}{8d^2}.
\]
By the definition of $u_{r,d}$, the last expression is at least $3r$.
Finally,
\[
 \frac{G_{r,d}(-y)}{H_{r,d}(-y)}
 =-\frac12\left(\frac{H_{r,d}(y)}{H_{r,d}(-y)}-1\right)
 \le-\frac{3r-1}{2}\le-r.
\]
\end{proof}

Set
\[
 F_{r,d}(x_1,\ldots,x_r)
 =\sum_{j=1}^rG_{r,d}(x_j)\prod_{i\ne j}H_{r,d}(x_i).
\]
Every multivariate Taylor coefficient of $F_{r,d}$ is nonnegative, and
\begin{equation}\label{eq:separator-factorisation}
 F_{r,d}(x_1,\ldots,x_r)
 =\left(\prod_{i=1}^rH_{r,d}(x_i)\right)
 \left(\sum_{j=1}^r\frac{G_{r,d}(x_j)}{H_{r,d}(x_j)}\right).
\end{equation}

\begin{lemma}\label{lem:separator-bounds}
Let $r\ge2$ and $d\ge3$ be integers. The following statements hold.
\begin{enumerate}[label=\textup{(\roman*)}]
\item Let $x=(x_1,\ldots,x_r)\in\R^r$. If $x_k\le-L_{r,d}$ for some
$k\in[r]$, then
$F_{r,d}(x_1,\ldots,x_r)\le-1.$

\item Let $A_1,\ldots,A_r\ge-1$, and put
$M=\max_{i\in[r]}A_i$. Then
\begin{equation}\label{eq:separator-growth}
 F_{r,d}(L_{r,d}A_1,\ldots,L_{r,d}A_r)
 \le D_r\exp\!\left(c_{r,d}(M+1)^{1/d}\right),
\end{equation}
where
$
 D_r=r2^{r-1}$ and $
 c_{r,d}=2ru_{r,d}.
$
\end{enumerate}
\end{lemma}

\begin{proof}
For every real $z$,
\begin{equation}\label{eq:H-at-least-one}
 H_{r,d}(z)\ge1.
\end{equation}
Suppose first that $x_k\le-L_{r,d}$ for some $k\in[r]$. By
\Cref{lem:ratio-estimates},
\[
 \frac{G_{r,d}(x_k)}{H_{r,d}(x_k)}\le-r.
\]
For every $j\ne k$, the same lemma gives
\[
 \frac{G_{r,d}(x_j)}{H_{r,d}(x_j)}\le\frac12.
\]
Consequently,
\[
 \sum_{j=1}^r\frac{G_{r,d}(x_j)}{H_{r,d}(x_j)}
 \le-r+\frac{r-1}{2}=-\frac{r+1}{2}\le-1.
\]
Using \eqref{eq:separator-factorisation} and
\eqref{eq:H-at-least-one}, we obtain
$
 F_{r,d}(x_1,\ldots,x_r)\le-1.
$
This proves part~\textup{(i)}.

For part~\textup{(ii)}, we first establish the one-variable estimate
\begin{equation}\label{eq:H-growth}
 H_{r,d}(L_{r,d}A)
 \le2\exp\!\left(2u_{r,d}(A+1)^{1/d}\right)
 \qquad(A\ge-1).
\end{equation}
Suppose first that $-1\le A\le0$. Define
$u=u_{r,d}(-A)^{1/d}$. Then $0\le u\le u_{r,d}$ and
$L_{r,d}A=-u^d$. By \eqref{eq:negative-axis-bound},
$
 a_{r,d}|E_d(-u^d)|^2
 \le e^{-2u_{r,d}\cos\theta_d}e^{2u\cos\theta_d}\le1.
$
Thus $H_{r,d}(L_{r,d}A)\le2$, which implies
\eqref{eq:H-growth} in this range.

Now suppose that $A\ge0$. Define $u=u_{r,d}A^{1/d}$, so that
$L_{r,d}A=u^d$. Equation \eqref{eq:positive-axis-upper} gives
$
 E_d(L_{r,d}A)\le e^{u_{r,d}A^{1/d}}.
$
Since $0<a_{r,d}\le1$,
\[
 H_{r,d}(L_{r,d}A)
 \le1+e^{2u_{r,d}A^{1/d}}
 \le2e^{2u_{r,d}(A+1)^{1/d}}.
\]
This proves \eqref{eq:H-growth} for all $A\ge-1$.

Put $\widetilde x_i=L_{r,d}A_i$. If
$F_{r,d}(\widetilde x_1,\ldots,\widetilde x_r)<0$, then
\eqref{eq:separator-growth} is immediate. We may therefore assume that this
quantity is nonnegative. By \eqref{eq:separator-factorisation},
\[
 0\le\sum_{i=1}^r
 \frac{G_{r,d}(\widetilde x_i)}{H_{r,d}(\widetilde x_i)}
 \le\frac r2.
\]
Using \eqref{eq:H-growth}, we obtain
\begin{align*}
 F_{r,d}(\widetilde x_1,\ldots,\widetilde x_r)
 &\le\frac r2\prod_{i=1}^rH_{r,d}(\widetilde x_i)\\
 &\le r2^{r-1}\exp\!\left(
2u_{r,d}\sum_{i=1}^r(A_i+1)^{1/d}\right)\\
 &\le r2^{r-1}\exp\!\left(2ru_{r,d}(M+1)^{1/d}\right).
\end{align*}
This is \eqref{eq:separator-growth}.
\end{proof}

\begin{theorem}\label{thm:correlation}
Let $r\ge2$ and $d\ge3$ be integers. Let $U$ and $U'$ be independent
copies of a random variable taking values in a finite set $\mathcal X$.
For each $i\in[r]$, let $\mathcal H_i$ be a real Hilbert space and let
$\sigma_i:\mathcal X\to\mathcal H_i$ be an arbitrary map. Set
$
 Z_i=\ip{\sigma_i(U)}{\sigma_i(U')}.
$
Then there exist $i\in[r]$ and $\lambda\ge-1$ such that
\begin{equation}\label{eq:correlation-tail}
 \Pp\bigl(
  Z_i\ge\lambda,\
  Z_j\ge-1\text{ for every }j\ne i
 \bigr)
 \ge
 \beta_r\exp\!\left(-C_{r,d}(\lambda+1)^{1/d}\right),
\end{equation}
where
\begin{equation*}
 \beta_r=\frac1{4D_r(r+1)},
 \qquad
 C_{r,d}=2c_{r,d}=4ru_{r,d}.
\end{equation*}
In particular, there are absolute constants $C_1,C_2>0$ such that
\begin{equation*}
 C_{r,d}\le C_1rd^2\log(2rd),
 \qquad
 \log(1/\beta_r)\le C_2r.
\end{equation*}
\end{theorem}

\begin{proof}
For brevity, write
$
 L=L_{r,d},
 c=c_{r,d},$ and $
 C=C_{r,d}.
$
Since $\mathcal X$ is finite, each $Z_i$ has finite range. In particular,
the random vector $(LZ_1,\ldots,LZ_r)$ is bounded.

Write the multivariate Taylor expansion of $F_{r,d}$ as
\[
 F_{r,d}(x_1,\ldots,x_r)
 =\sum_{\mathbf a\in\mathbb Z_{\ge0}^{\,r}}
 q_{\mathbf a}\prod_{i=1}^r x_i^{a_i},
 \qquad q_{\mathbf a}\ge0,
\]
where $\mathbf a=(a_1,\ldots,a_r)$ and
$|\mathbf a|=a_1+\cdots+a_r$. To justify taking expectations term by term,
set
\[
 B=\max\left\{
 1,
 L\max_{\substack{i\in[r]\\x,y\in\mathcal X}}
 \left|\ip{\sigma_i(x)}{\sigma_i(y)}\right|
 \right\}.
\]
Then $|LZ_i|\le B$ almost surely for every $i\in[r]$. Since $F_{r,d}$ is
entire and its Taylor coefficients are nonnegative,
\[
 \sum_{\mathbf a\in\mathbb Z_{\ge0}^{\,r}}
 q_{\mathbf a}L^{|\mathbf a|}
 \E\left(\left|\prod_{i=1}^rZ_i^{a_i}\right|\right)
 \le
 \sum_{\mathbf a\in\mathbb Z_{\ge0}^{\,r}}
 q_{\mathbf a}B^{|\mathbf a|}
 =F_{r,d}(B,\ldots,B)<\infty.
\]
Thus the Taylor series is absolutely summable in $L^1$, so expectation may
be interchanged with summation. By \Cref{lem:positive-moments},
\begin{equation}\label{eq:separator-positive-expectation}
 \E\bigl(F_{r,d}(LZ_1,\ldots,LZ_r)\bigr)
 =\sum_{\mathbf a}q_{\mathbf a}L^{|\mathbf a|}
 \E\left(\prod_{i=1}^rZ_i^{a_i}\right)
 \ge0.
\end{equation}

Let
$\mathcal E=\{Z_j\ge-1\text{ for every }j\in[r]\}$, $\bar{\mathcal E}=\{Z_j<-1\text{ for some }j\in[r]\}$
and
 $M=\max_{j\in[r]}Z_j$,
and define
\[
 Y=\bigl((M+1)_+\bigr)^{1/d},
 \qquad
 (t)_+=\max\{t,0\}.
\]
Thus $Y\ge0$, and on $\mathcal E$ we have $Y=(M+1)^{1/d}$.

If $\bar{\mathcal E}$ occurs, then $Z_k<-1$ for some $k\in[r]$, and hence
$LZ_k<-L$. Therefore, \Cref{lem:separator-bounds}(i) gives
\[
 F_{r,d}(LZ_1,\ldots,LZ_r)\le-1
 \qquad\text{on }\bar{\mathcal E}.
\]
On the other hand, on $\mathcal E$ all coordinates $Z_j$ are at least
$-1$. Applying \Cref{lem:separator-bounds}(ii) with $A_j=Z_j$ gives
\[
 F_{r,d}(LZ_1,\ldots,LZ_r)\le D_re^{cY}
 \qquad\text{on }\mathcal E.
\]
Consequently,
\[
 F_{r,d}(LZ_1,\ldots,LZ_r)
 \le D_re^{cY}\1_{\mathcal E}-\1_{\bar{\mathcal E}}.
\]
Taking expectations and using
\eqref{eq:separator-positive-expectation}, we obtain
\begin{equation}\label{eq:correlation-master}
 D_r\E\bigl[e^{cY}\1_{\mathcal E}\bigr]
 \ge1-\Pp(\mathcal E).
\end{equation}

For $i\in[r]$ and $t\ge-1$, put
$
 \mathcal A_i(t)
 =\{Z_i\ge t,\ Z_j\ge-1\text{ for every }j\ne i\}.
$
If $\Pp(\mathcal E)\ge\beta_r$, choose $\lambda=-1$. Then
$\mathcal A_i(-1)=\mathcal E$ for every $i\in[r]$, while
$(\lambda+1)^{1/d}=0$. Hence \eqref{eq:correlation-tail} follows.

It remains to consider the case $\Pp(\mathcal E)<\beta_r$. Suppose, for a
contradiction, that \eqref{eq:correlation-tail} fails for every $i\in[r]$
and every $\lambda\ge-1$. Thus
\begin{equation}\label{eq:correlation-failure}
 \Pp\bigl(\mathcal A_i(\lambda)\bigr)
 <\beta_re^{-C(\lambda+1)^{1/d}}
\end{equation}
for every $i\in[r]$ and every $\lambda\ge-1$.

Fix $u\ge0$ and set $\lambda_u=u^d-1$. Then $\lambda_u\ge-1$ and
$(\lambda_u+1)^{1/d}=u$. Moreover, on $\mathcal E$,
 $Y\ge u$ if and only if
 $M\ge\lambda_u$.
Therefore
\[
 \mathcal E\cap\{Y\ge u\}
 =\bigcup_{i=1}^r\mathcal A_i(\lambda_u).
\]
The union bound and \eqref{eq:correlation-failure} yield
\begin{equation*}
 \Pp\bigl(\mathcal E\cap\{Y\ge u\}\bigr)
 <r\beta_re^{-Cu}.
\end{equation*}

For every $y\ge0$,
$e^{cy}=1+c\int_0^\infty e^{cu}\1_{\{y>u\}}\,\dd u.$
Multiplying by $\1_{\mathcal E}$, taking expectations, and applying
Tonelli's theorem gives
\begin{align*}
 \E\bigl[e^{cY}\1_{\mathcal E}\bigr]
 &=\Pp(\mathcal E)
 +c\int_0^\infty e^{cu}
 \Pp\bigl(\mathcal E\cap\{Y>u\}\bigr)\,\dd u\\
 &\le\Pp(\mathcal E)
 +c\int_0^\infty e^{cu}
 \Pp\bigl(\mathcal E\cap\{Y\ge u\}\bigr)\,\dd u\\
 &<\beta_r+r\beta_rc
 \int_0^\infty e^{-(C-c)u}\,\dd u\\
 &=\beta_r+r\beta_r\frac{c}{C-c}
 =(r+1)\beta_r
 =\frac1{4D_r}.
\end{align*}
Here we used $C=2c>c>0$. Consequently,
$D_r\E\bigl[e^{cY}\1_{\mathcal E}\bigr]<\frac14$.
On the other hand, \eqref{eq:correlation-master} and
$\Pp(\mathcal E)<\beta_r$ give
$D_r\E\bigl[e^{cY}\1_{\mathcal E}\bigr]
 >1-\beta_r$.
Since $r\ge2$ and $D_r=r2^{r-1}\ge4$,
\[
 \beta_r=\frac1{4D_r(r+1)}\le\frac1{48}<\frac14,
\]
and hence $1-\beta_r>3/4$, a contradiction. This proves
\eqref{eq:correlation-tail}.

Finally, \eqref{eq:urd-bound} and $C_{r,d}=4ru_{r,d}$ give
$C_{r,d}\le4C_0rd^2\log(2rd),$
so one may take $C_1=4C_0$. Moreover,
\[
 \log(1/\beta_r)
 =\log\bigl(4r2^{r-1}(r+1)\bigr)
 =(r+1)\log2+\log r+\log(r+1)
 \le3r
\]
for every $r\ge2$, so one may take $C_2=3$.
\end{proof}

\section{Monochromatic books}\label{sec:book}

Let $r\ge2$ and $d\ge3$ be integers, and let $0<\beta\le1$ and $C>0$.
Write $\mathcal G_r^{(d)}(\beta,C)$ for the following assertion. For every finite set $\mathcal X$, every pair $U,U'$ of
independent copies of a random variable taking values in $\mathcal X$, and
every collection of maps $\sigma_i:\mathcal X\to\mathcal H_i$ into real
Hilbert spaces, set
$
 Z_i=\ip{\sigma_i(U)}{\sigma_i(U')}
$, where $i\in[r].
$
Then there exist $i\in[r]$ and $\lambda\ge-1$ such that
\[
 \Pp\bigl(
  Z_i\ge\lambda,
  Z_j\ge-1\text{ for every }j\ne i
 \bigr)
 \ge\beta\exp\!\left(-C(\lambda+1)^{1/d}\right).
\]
Thus \Cref{thm:correlation} states that
$\mathcal G_r^{(d)}(\beta_r,C_{r,d})$ holds.

\begin{lemma}\label{cap-lem:partial}
Let $r\ge2$, $d\ge3$, and assume $\mathcal G_r^{(d)}(\beta,C)$. Let $I\subseteq[r]$, let $X$ be
nonempty, and for $i\in I$ let $Y_i$ be nonempty with
$p_i=p_i(X,Y_i)>0$ and $\alpha_i>0$. There are $x\in X$,
$\ell\in[r]$, $\lambda\ge-1$, nonempty $X'\subseteq X$, and,
only for $i\in I$, nonempty $Y_i'\subseteq N_i(x)\cap Y_i$ such that
\begin{align*}
 |X'|&\ge\beta e^{-C(\lambda+1)^{1/d}}|X|,\\
 |Y_i'|&=p_i|Y_i|,\qquad
 p_i(X',Y_i')\ge p_i-\alpha_i\quad(i\in I).
\end{align*}
If $\ell\in I$, then
$p_\ell(X',Y_\ell')\ge p_\ell+\lambda\alpha_\ell$.
If $\ell\notin I$, then $\lambda\le0$.
\end{lemma}

\begin{proof}
For $i\in I$, the number $m_i=p_i|Y_i|$ is a positive integer.
For each $v\in X$ choose $A_i(v)\subseteq N_i(v)\cap Y_i$ of
size $m_i$, and set
\[
 \sigma_i(v)=\frac{\mathbf1_{A_i(v)}-p_i\mathbf1_{Y_i}}
 {\sqrt{\alpha_i p_i|Y_i|}}.
\]
For $i\notin I$ use the identically zero map into $\mathbb R$.
For $i\in I$,
\[
 \langle\sigma_i(u),\sigma_i(v)\rangle
 =\frac{|A_i(u)\cap A_i(v)|-p_i^2|Y_i|}{\alpha_i p_i|Y_i|}.
\]
Apply the correlation inequality and choose a fiber in the first variable
of at least its average size. Let $x$ be that variable, let $X'$ be
the fiber, and put $Y_i'=A_i(x)$. The displayed identity, and
$A_i(v)\subseteq N_i(v)$, give all active density inequalities.
If the distinguished map is zero, the positive probability in
the correlation inequality forces $\lambda\le0$.
When $I=\varnothing$, one may simply take $X'=X$, arbitrary $x,\ell$,
and $\lambda=-1$.
\end{proof}

\begin{lemma}\label{cap-lem:book}
Let $r\ge2$ and $d\ge3$ be integers. Let $I\subseteq[r]$, and prescribe integer targets $t_i\ge2$ for
all $r$ colors. Set $S=\sum_i t_i$, $T=\max_i t_i$, and $\rho=\log(r/\beta)$. Let $X$ be
nonempty and, for $i\in I$, let $Y_i$ be nonempty with
$p_i=p_i(X,Y_i)>0$. Assume $\mathcal G_r^{(d)}(\beta,C)$, and let
$0<\delta\le1/4$, with $\delta\le p_i/4$ for every $i\in I$.
Define $\Pi_i$ and $\Xi$ by either of the following choices.

\smallskip\noindent
\textup{(a) Boosts allowed.} Put $L=\log(1/\delta)$,
$\tau=(d-1)/d$. Require
\[
 \lambda_0\ge\max\{2,6L\},\qquad
 t_i\ge\lambda_0\delta^{-1/(d-1)}\quad(i\in I),
\]
and use
\begin{equation*}
 \Pi_i=\frac{3\delta}{p_i}+
       \frac{6L\log(1/p_i)}{\lambda_0},\qquad
 \Xi=2\rho+4C\lambda_0^{1/d}+\frac{12CL}{\lambda_0^\tau}.
\end{equation*}
\textup{(b) No boosts.} No additional length condition is imposed. Use
\begin{equation*}
 \Pi_i=\frac{3\delta}{p_i},\qquad
 \Xi=\rho+C(1+4T/\delta)^{1/d}.
\end{equation*}
In either case, suppose
\begin{equation}\label{cap-eq:bookbudget}
 |X|\ge2S e^{S\Xi},\qquad
 p_i^{t_i}e^{-\Pi_i t_i}|Y_i|\ge1\quad(i\in I).
\end{equation}
Then there is either a color-$j$ clique $Q\subseteq X$ of order
$t_j$ for some $j\notin I$, or a color-$i$ clique $Q\subseteq X$
of order $t_i$ for some $i\in I$, with a page
$P\subseteq Y_i\setminus Q$ joined entirely in color $i$ to $Q$ and
\begin{equation*}
 |P|\ge p_i^{t_i}e^{-\Pi_i t_i}|Y_i|.
\end{equation*}
\end{lemma}

\begin{proof}
Construct temporary spines $Q_i$, initially empty. The current reservoir
is joined to each $Q_i$ in color $i$; for $i\in I$, page $i$ is also
joined to $Q_i$. At a current state write
\[
 \kappa_i=p_i(X(s),Y_i(s))-p_i+\delta\quad(i\in I),
 \qquad \alpha_i=\kappa_i/t_i,
\]
and apply Lemma~\ref{cap-lem:partial}. For an extension, choose a color $j\in[r]$ satisfying
\[
 |N_j(x)\cap X'|\ge(|X'|-1)/r.
\]
Add $x$ to $Q_j$ and replace the reservoir by $N_j(x)\cap X'$.
If $j\in I$, also replace page $j$ by $Y_j'$; if $j\notin I$,
no page changes. In (a), extend when $\lambda\le\lambda_0$ and
otherwise boost: replace the reservoir by $X'$ and page $\ell$ by
$Y_\ell'$, without extending a spine. Such a boost always has
$\ell\in I$ by Lemma~\ref{cap-lem:partial}. In (b), always extend.
Stop at the first $|Q_i|=t_i$. Every selected vertex belongs to the
original $X$, and looplessness preserves all required disjointness.

At any step before stopping, let $e_i=|Q_i|$, and let $\mathcal B_i$
be the set of preceding boosts in color $i$. Restricting a reservoir cannot reduce a
minimum density. Hence
\begin{equation}\label{cap-eq:clock}
 \kappa_i\ge\delta(1-1/t_i)^{e_i}
       \prod_{v\in\mathcal B_i}(1+\lambda(v)/t_i)
 \ge\delta/4\quad(i\in I).
\end{equation}
The current active density is therefore at least
$p_i-3\delta/4>0$. If the distinguished index is active and
$\lambda>0$, its updated density on the nonempty $X'$ is at most one,
so already before the majority restriction,
\begin{equation*}
 \lambda\le t_\ell/\kappa_\ell\le4T/\delta.
\end{equation*}
An inactive distinguished index has $\lambda\le0$ instead.

In (b), fewer than $S$ extensions occur up to and including stopping.
Each retains at least $e^{-\Xi}$ of the reservoir, apart from an
additive loss at most one. Thus every such prefix, including a proposed
last extension, retains at least
\begin{equation}\label{cap-eq:reservoir}
 e^{-S\Xi}|X|-S\ge S>0.
\end{equation}
Each active page changes at most $t_i$ times. Since
$\log(1-u)\ge-2u$ for $0\le u\le3/16$, it has size at least
\[
 (p_i-3\delta/4)^{t_i}|Y_i|
 \ge p_i^{t_i}e^{-3\delta t_i/p_i}|Y_i|.
\]

In (a), $\kappa_i\le1+\delta\le5/4$ for $i\in I$.
Equation~\eqref{cap-eq:clock} gives
\[
 \sum_{v\in\mathcal B_i}\log(1+\lambda(v)/t_i)\le3L,
 \qquad |\mathcal B_i|\le6Lt_i/\lambda_0\le t_i.
\]
Here $t_i\ge\lambda_0$, so
$\log(1+\lambda_0/t_i)\ge\lambda_0/(2t_i)$.
For any boost in color $i$, split at $\lambda=t_i$. Below this
value, use $\log(1+\lambda/t_i)\ge\lambda/(2t_i)$; above it,
use $\lambda+1\le5t_i/\delta$ and $\log(1+\lambda/t_i)\ge\log2$.
Together with the active length condition, these imply
\[
 \frac{(\lambda+1)^{1/d}}{\log(1+\lambda/t_i)}
 \le\frac{4t_i}{\lambda_0^\tau},\qquad
 \sum_{\text{boosts}}(\lambda+1)^{1/d}
 \le\frac{12LS}{\lambda_0^\tau}.
\]
In the second case the numerical inequality needed is
$5^{1/3}/\log2<4$. There are fewer than $S$ extensions and at most
$\sum_{i\in I}t_i\le S$ boosts. Multiplying their reservoir factors
gives logarithmic loss at most $S\Xi$. The additive extension losses
sum to at most $S$, since subsequent retention factors are at most one.
Thus \eqref{cap-eq:reservoir} also holds in (a).

An active page changes at most $t_i+6Lt_i/\lambda_0\le2t_i$ times,
so its size is at least
\[
 p_i^{t_i+6Lt_i/\lambda_0}
 (1-3\delta/(4p_i))^{2t_i}|Y_i|
 \ge p_i^{t_i}e^{-\Pi_i t_i}|Y_i|.
\]
These estimates ensure that each step is well defined. The active
densities are positive by \eqref{cap-eq:clock}, and the bound on $\lambda$
holds on $X'$ before an extension. Equation~\eqref{cap-eq:reservoir}
keeps the reservoir nonempty after an extension. A boost replaces it
by the nonempty set $X'$, and the bounds on $\kappa_i$ and the number
of boosts remain valid after the update. The page estimates and
\eqref{cap-eq:bookbudget} exclude a first empty page. Since the numbers
of extensions and boosts are bounded, some $Q_i$ reaches its target.
For $i\notin I$ this is the required clique; for $i\in I$, take $Q_i$
and its page. If $I=\varnothing$, only extensions occur and the same
reservoir estimate applies.
\end{proof}

\section{Proof of the main theorems}\label{sec:capacity}

Set $K=k-1$. For positive integer capacities $b_i\le K$ define
\[
 B=\sum_i b_i,\quad w_i=b_i/B,\quad
 \M(\bb)=\frac{B!}{\prod_i b_i!},\quad F(\bb)=rK-B,
 \quad \widehat\Psi(W,\bb)=\log\frac{|W|}{\whM(\bb)}.
\]
By Lemma~\ref{lem:factorial}, a nonempty reservoir $W$ avoiding
a color-$i$ clique of size $b_i+1$ for every $i$ satisfies
\begin{equation}\label{cap-eq:negative}
 \widehat\Psi(W,\bb)<0.
\end{equation}
The factor $2/r!$ cancels when taking differences of $\widehat\Psi$.

\begin{lemma}\label{cap-lem:profile}
Let $|W|=N$, and suppose $q,a>0$ and
\[
 |N_i(v)\cap W|/N\le w_i+aq/B\quad(v\in W,\ i\in[r]).
\]
Put $D=arq+B/N$. There are $X\subseteq W$ and $h_i\ge0$ with
\begin{equation}\label{cap-eq:profile}
 |X|\ge8^{-r}N,\qquad \sum_i b_i h_i\le2D,
 \qquad h_i\le D/b_i,\qquad p_i(X,W)\ge w_i(1-h_i).
\end{equation}
\end{lemma}

\begin{proof}
Write $d_i(v)=|N_i(v)\cap W|/N$ and
$z_i(v)=(b_i-Bd_i(v))_+$. Since $\sum_i d_i(v)=1-1/N$,
\[
 \sum_i z_i(v)=B\sum_i(d_i(v)-w_i)_++B/N\le D.
\]
The vector $v\mapsto(\lceil rz_i(v)/D\rceil)_{i=1}^r$ has
entries between zero and $r$ and sum at most $2r$.
There are at most $\binom{3r}{r}\le8^r$ possible vectors. On a
largest common class with vector $(v_i)$, put $h_i=Dv_i/(rb_i)$.
Upward rounding proves all assertions simultaneously on this one class.
\end{proof}

\begin{lemma}\label{cap-lem:assignment}
Suppose $q\ge2$, $a\le1/192$, and the profile satisfies $D\le2arq$.
Set
\begin{equation}\label{cap-eq:assignment}
 u_i=\lceil q+24b_i h_i\rceil,\qquad
 I=\{i:u_i\le b_i/8\},\qquad
 t_i=\begin{cases}u_i,&i\in I,\\ b_i+1,&i\notin I.\end{cases}
\end{equation}
Then all $t_i\ge2$, and
\begin{equation}\label{cap-eq:assignmentbounds}
 S:=\sum_i t_i\le24rq,\qquad
 b_i\ge8q,\quad h_i\le1/192,\quad q\le t_i\le b_i/8\quad(i\in I).
\end{equation}
\end{lemma}

\begin{proof}
The weighted sum gives
\[
 \sum_i u_i\le rq+48D+r\le(1+96a)rq+r\le2rq.
\]
If $i\notin I$, the integers satisfy $b_i<8u_i$, hence
$b_i+1\le8u_i$. Thus $S\le8\sum_i u_i\le16rq\le24rq$. If $i\in I$, use
$q+24b_i h_i\le t_i\le b_i/8$ to obtain the remaining bounds.
\end{proof}

\begin{lemma}\label{cap-lem:geometry}
Let $r\ge16$, $K\ge4$, $H=rK\log r$, and $F(\bb)<rK/4$.
Then
\begin{equation*}
 \max_i w_i\le1/4,\qquad \log\M(\bb)\ge H/32.
\end{equation*}
For every $i$ and integer $2\le t\le b_i/8$,
\begin{equation}\label{cap-eq:curvature}
 \frac{\M(\bb-t\ee_i)}{\M(\bb)}=w_i^te^{-D_i(t)},\qquad
 D_i(t)\ge\frac{t^2}{12b_i}.
\end{equation}
\end{lemma}

\begin{proof}
Fewer than $r/2$ capacities can be below $K/2$, because each such
capacity contributes more than $K/2$ to $F$. Choose
$m=\lceil r/2\rceil$ capacities at least $K/2$. In particular
$B\ge rK/4$ and $w_i\le4/r\le1/4$.
Count multiset words whose initial segment consists of
$\lfloor K/2\rfloor$ arbitrary permutation blocks of these $m$
letters, with all remaining letters in a fixed order. Then
\[
 \M(\bb)\ge(m!)^{\lfloor K/2\rfloor},\qquad
 \log\M(\bb)\ge\frac K4\frac m2\log m\ge H/32.
\]
Here $m!\ge m^{m/2}$ follows by pairing
$j(m+1-j)\ge m$, and $\log m\ge\tfrac12\log r$.
Factorial cancellation gives
\[
 D_i(t)=\sum_{j=0}^{t-1}
 [\log(1-j/B)-\log(1-j/b_i)]
 \ge\frac{t(t-1)}2\frac{1-w_i}{b_i}
 \ge\frac{t^2}{12b_i}.
\]
For the first inequality integrate $1/(1-x)\ge1$ over the interval
between $j/B$ and $j/b_i$; for the last use $t-1\ge t/2$ and
the weaker $w_i\le2/3$.
\end{proof}

Fix absolute constants
\begin{equation}\label{cap-eq:constants}
 0<a\le1/192,\qquad 0<\zeta\le1/8,\qquad A\ge12,\qquad
 3\zeta/2+6/A<1/48.
\end{equation}
For $r\ge16$ set
\begin{equation*}
 s=\log(2r),\quad d=\lceil s\rceil,\quad
 \theta=\frac{b}{rs^2},\quad q=\theta K,\quad
 \delta=\frac{\zeta\theta}{4r},\quad
 \Gamma=\frac{a\theta r}{64}=\frac{ab}{64s^2}.
\end{equation*}
The constant $b>0$ will be chosen separately in two ranges of $K$.

\begin{lemma}\label{cap-lem:availability}
There are absolute $b_{\rm H},K_{\rm H}>0$ for which the assertions
below hold for $r\ge16$ and $K\ge K_{\rm H}r^2s^6$, using
$b=b_{\rm H}$, the constants $(\beta_r,C_{r,d})$ of
Theorem~\ref{thm:correlation}, and Lemma~\ref{cap-lem:book}(a).
After fixing these constants there are absolute
$b_{\rm L},K_{\rm L}>0$ for which the same assertions hold for
\[
 K_{\rm L}rs^2\le K\le K_{\rm H}r^2s^6,
\]
using $b=b_{\rm L}$ and Lemma~\ref{cap-lem:book}(b).

Consider any state with $F<rK/4$, $N\ge1/(a\theta)$, and the
profile of Lemma~\ref{cap-lem:profile}. Use \eqref{cap-eq:assignment} and let
$p_i\ge w_i(1-h_i)$ be actual minimum densities for $i\in I$.
Then $K\ge4$, $q\ge2$, \eqref{cap-eq:assignmentbounds} holds, and all
parameter conditions of the indicated mixed book lemma hold. Moreover,
\begin{equation}\label{cap-eq:scaledloss}
 p_i\ge4\theta/r>0,\qquad
 \Pi_i\le q/(48b_i)\quad(i\in I).
\end{equation}
With $H=rK\log r$, we have
\begin{equation}\label{cap-eq:budgets}
 \begin{gathered}
 S\Xi\le H/256,\qquad r\log8+\log(2S)\le H/256,\\
 t_i\log(1/p_i)+\Pi_i t_i\le H/256\quad(i\in I),\\
 \log(1/(a\theta))\le H/256,\qquad \Gamma k\le H/256.
 \end{gathered}
\end{equation}
\end{lemma}

\begin{proof}
Since $B\le rK$ and $N^{-1}\le a\theta$, one has $D\le2arq$.
Lemma~\ref{cap-lem:assignment} applies once $q\ge2$. For each active
color,
\begin{equation}\label{cap-eq:densityscale}
 p_i\ge b_i/(2B)\ge4\theta/r,\qquad
 \frac{3\delta}{p_i}\le\frac{3\zeta\theta B}{2rb_i}
 \le\frac{3\zeta}{2}\frac q{b_i}.
\end{equation}
In particular $\delta\le p_i/4$. Choose $b\le1$ sufficiently small
that $a\theta\le1$ and $\delta\le1/4$. Put $L=\log(1/\delta)$.
The active lower bound implies $\log(1/p_i)\le L$.

In the high range put $\lambda_0=AL^2/\theta$ and
$B_0=1+\log(1/b)$. Uniformly for $r\ge16$,
$d\ge4$, $d\ge s$, $d-1\ge s/2$, and
$C=C_{r,d}=O(rs^3)$, $\rho=\log(r/\beta_r)=O(r)$. There are absolute constants with
\[
 1\le L\le C_1B_0s,\qquad
 \lambda_0^{1/d}\le C_2(AB_0^2/b)^{1/4}.
\]
For the second bound, use $r^{1/d}\le e$ and the uniform bound on
$s^{4/d}$. Also $\tau=(d-1)/d\ge3/4$ gives
\[
 L/\lambda_0^\tau=A^{-\tau}\theta^\tau L^{1-2\tau}\le1.
\]
Consequently
\begin{equation*}
 \Xi\le C_3rs^3[1+(AB_0^2/b)^{1/4}],\qquad
 \frac{24rq\Xi}{H}
 \le C_4[b+A^{1/4}b^{3/4}B_0^{1/2}].
\end{equation*}
The latter tends to zero as $b\downarrow0$. Choose $b_{\rm H}$
small enough to satisfy the first inequality in \eqref{cap-eq:budgets}
and the remaining estimates below. For this fixed coefficient,
$\delta^{-1/(d-1)}=O_b(1)$ and $\lambda_0=O_b(rs^4)$.
Thus, by choosing $K_{\rm H}$ sufficiently large,
\[
 q\ge\lambda_0\delta^{-1/(d-1)}\ge2
 \quad\text{whenever }K\ge K_{\rm H}r^2s^6.
\]
This lower bound is needed only for active $t_i\ge q$.
Also $\lambda_0\ge\max\{2,6L\}$, and
\[
 \frac{6L\log(1/p_i)}{\lambda_0}\le6\theta/A
 \le(6/A)q/b_i.
\]
Combining this with \eqref{cap-eq:densityscale} and
\eqref{cap-eq:constants} proves \eqref{cap-eq:scaledloss} in the high range.

For the low range first keep $K_{\rm H}$ fixed. With $T\le S\le24rq$,
\begin{equation*}
 1+4T/\delta\le1+384r^2K/\zeta
 \le1+(384K_{\rm H}/\zeta)r^4s^6.
\end{equation*}
Its $1/d$ power is bounded by a constant depending only on the
already fixed $K_{\rm H},\zeta$, and is independent of the new $b$.
Thus $\Xi\le C_5rs^3$ in case (b). Choose $b_{\rm L}$
small enough that $24rq\Xi/H\le1/256$ and the remaining inequalities
in \eqref{cap-eq:budgets} hold; then choose $K_{\rm L}$ large enough that
$q\ge2$ and $K\ge4$. Equation~\eqref{cap-eq:densityscale} alone proves
\eqref{cap-eq:scaledloss} in this range.

To verify the remaining inequalities in \eqref{cap-eq:budgets}, note that
$q/b_i\le1/8$ for each active color, so
\begin{equation*}
 t_i\log(1/p_i)+\Pi_i t_i
 \le24rq\,[L+1/384].
\end{equation*}
Its ratio to $H$ tends uniformly to zero with $b$, because
$L\le C_1B_0s$ and $bB_0\to0$.
Similarly, $k=K+1\le2K$ gives
$\Gamma k/H\le a\theta/(32\log r)$, uniformly small with $b$.
After fixing $b$, we can increase the relevant threshold constant so
that $r\log8+\log(48r\theta K)$ and $\log(1/(a\theta))$ are at most $H/256$.
Indeed $r\theta=b/s^2\le1$, so the first is at most
$r\log8+\log(48K)$; its ratio to $rK\log r$ tends uniformly to zero
under $K\ge K_*rs^2$ as $K_*\to\infty$, using the decrease of
$\log K/K$ for $K\ge3$. The second is $O_b(s)$ and has the same
property. 
We choose $(a,\zeta,A)$ first, followed by $b_{\rm H}$, $K_{\rm H}$,
$b_{\rm L}$, and $K_{\rm L}$, in that order. All constants are absolute.
\end{proof}

\begin{proof}[Proof of Theorem~\ref{thm:large} for $r\ge16$]
Fix the constants of Lemma~\ref{cap-lem:availability} in their stated
order, and put $g=a\min\{b_{\rm H},b_{\rm L}\}/64$. Choose an absolute
$C_0$ so that, whenever $k\ge C_0rs_r^2$, the value $K=k-1$ lies in
one of these two ranges and satisfies $K\ge256$ and $\Gamma K\ge1$
for either choice of $b$.
Choose the high range when $K\ge K_{\rm H}r^2s_r^6$, and the low range
otherwise. In particular $g/s_r^2\le\Gamma$.
Set
\[
 x=\frac{2}{r!}M_r(k)e^{-gk/s_r^2},\qquad n=\lfloor x\rfloor,
 \qquad H=rK\log r.
\]
At the diagonal state Lemma~\ref{cap-lem:geometry} and
\eqref{cap-eq:budgets} give
\[
 \log x\ge\frac{H}{32}-\log(r!/2)-\Gamma k
 \ge\frac{6H}{256},
\]
because $\log(r!/2)\le r\log r\le H/256$. Thus $x\ge2$ and
$\log n\ge\log x-\log2$.

Suppose an $r$-coloring of $K_n$ contains no monochromatic $K_k$.
Initially set all $C_i=\varnothing$ and $W=V(K_n)$. At every stage the
sets $C_1,\ldots,C_r,W$ are pairwise disjoint, each $C_i$ is a
color-$i$ clique, and every edge from $C_i$ to $W$ has color $i$.
Write $N=|W|$, $b_i=K-|C_i|$, $B=\sum_i b_i$, and $F=rK-B$.
We maintain positive capacities and a nonempty reservoir. A color-$i$
clique of size $b_i+1$ in $W$ would complete $C_i$ to a $K_k$.
Consequently \eqref{cap-eq:negative} applies, whereas initially
\begin{equation}\label{eq:scaled-initial}
 \widehat\Psi\ge-\Gamma k-\log2.
\end{equation}

We perform the operations below while $F<rK/4$. Their positive gains
will give inductively, before each operation,
\begin{align}
 \log N
 &=\log\mathcal M(\bb)-\log(r!/2)+\widehat\Psi\notag\\
 &\ge\frac{H}{32}-\frac{H}{256}-\frac{H}{256}-\log2
 \ge\frac{5H}{256}.
 \label{eq:scaled-available}
\end{align}
Here $H\ge256\log2$. This estimate holds initially by
\eqref{eq:scaled-initial}; later it follows from the gains already made
and Lemma~\ref{cap-lem:geometry}, since $F<rK/4$. In particular it verifies
$N\ge1/(a\theta)$ before Lemma~\ref{cap-lem:availability} is used.

Suppose first that some $v\in W$ and $i\in[r]$ satisfy
\[
 |N_i(v)\cap W|>(w_i+aq/B)N,
 \qquad w_i=b_i/B.
\]
Add $v$ to $C_i$ and replace $W$ by $N_i(v)\cap W$. The new reservoir
is nonempty. If $b_i=1$, any vertex of the new reservoir together with
$C_i$ and $v$ gives a color-$i$ $K_k$, a contradiction. Otherwise
decrease $b_i$ by one. The exact multinomial ratio is
$\mathcal M(\bb-\ee_i)/\mathcal M(\bb)=w_i$, and hence
\begin{equation*}
 \Delta\widehat\Psi>\log(1+aq/b_i)
 \ge\log(1+a\theta)\ge\frac{a\theta}{2}
 \ge\frac{a\theta}{4}\Delta F,
 \qquad\Delta F=1.
\end{equation*}
We used $b_i\le K$ and $a\theta\le1$. The new capacity remains positive,
and all structural conditions are preserved.

If no such pair exists, Lemma~\ref{cap-lem:profile} gives one class
$X\subseteq W$ and deficits $h_i$ satisfying
\eqref{cap-eq:profile}. Assign the targets and active set by
\eqref{cap-eq:assignment}. For each active color initialize its page as
the entire old $W$ and use its actual minimum density
$p_i=p_i(X,W)\ge w_i(1-h_i)$.
Lemma~\ref{cap-lem:availability} verifies the length and density
conditions of Lemma~\ref{cap-lem:book}. Its reservoir requirement follows
from \eqref{eq:scaled-available}, since
\[
 r\log8+\log(2S)+S\Xi\le\frac{2H}{256}<\log N.
\]
For each active page, the corresponding requirement follows from
$t_i\log(1/p_i)+\Pi_i t_i\le H/256<\log N$.
The factor $8^{-r}$ occurs only in the bound for $|X|$; each initial
page is $W$.

If a terminal color $i\notin I$ completes its target, its clique has
size $b_i+1$ inside the old $W$. Together with the old $C_i$ it forms
a color-$i$ $K_k$, a contradiction. Otherwise an active color $i$ gives
a spine $Q$ of size $t_i$ and a page $P$ with
\[
 |P|\ge w_i^{t_i}
 \exp\!\left[-\left(2h_i+\frac{q}{48b_i}\right)t_i\right]N.
\]
This uses $\log(1-h_i)\ge-2h_i$ and
\eqref{cap-eq:scaledloss}. By \eqref{cap-eq:curvature} and the definition
of $t_i$,
\[
 D_i(t_i)\ge\frac{t_i^2}{12b_i}
 \ge\left(\frac{q}{12b_i}+2h_i\right)t_i.
\]
Retain this spine, replace $C_i$ by $C_i\cup Q$, replace $W$ by $P$,
and decrease $b_i$ by $t_i$. The new capacity is positive because
$t_i\le b_i/8$. All old outer cliques remain joined to the new
reservoir in their respective colors. Since only $b_i$ has changed,
the factorial ratio gives
\begin{equation*}
 \Delta\widehat\Psi\ge\frac{qt_i}{16b_i}
 \ge\frac{\theta t_i}{16}
 \ge\frac{a\theta}{4}\Delta F,
 \qquad\Delta F=t_i.
\end{equation*}
Here $a\le1/4$.

Both operations increase $\widehat\Psi$, so \eqref{eq:scaled-available}
holds at the next step whenever $F<rK/4$.
Every continuing operation strictly decreases the positive integer $B$.
Thus either a forbidden clique occurs or a first state with $F\ge rK/4$
is reached. The last operation is valid because its initial state
satisfies $F<rK/4$, and it leaves positive capacities and a nonempty reservoir. At this state,
\[
 \widehat\Psi\ge-\Gamma k-\log2+\frac{a\theta}{4}F
 \ge-\Gamma k-\log2+\frac{a\theta rK}{16}
 =\Gamma(3K-1)-\log2>0.
\]
This contradicts \eqref{cap-eq:negative}. Every coloring of $K_n$
therefore contains the target, so $R_r(k)\le n\le x$.
\end{proof}

\begin{proof}[Completion of Theorem~\ref{thm:large}]
It remains to cover $2\le r<16$. The fixed-color exponential improvement
of Balister et al.\ \cite{BalisterEtAl2026} gives, with
$\delta_r=2^{-160}r^{-12}$,
\[
 R_r(k)\le r^{rk}e^{-\delta_r k}
 \qquad(k\ge2^{200}r^{20}).
\]
Among the multinomial coefficients with total degree $rK$, the balanced
one is maximal: transferring a unit from an entry $u$ to an entry $v$
with $u\ge v+2$ increases the coefficient by $u/(v+1)>1$.
There are at most $(rK+1)^r$ nonnegative composition vectors, and their
coefficients sum to $r^{rK}$. Hence
\[
 M_r(k)\ge\frac{r^{rK}}{(rK+1)^r},\qquad
 \log\frac{r^{rk}}{(2/r!)M_r(k)}
 \le r\log r+r\log(rK+1)+\log(r!/2)=O_r(\log k).
\]
For each of these finitely many $r$, increase a threshold until this
last expression is at most $\delta_r k/2$ and the preceding
Ramsey estimate applies. The result is
\[
 R_r(k)\le\frac2{r!}M_r(k)e^{-\delta_r k/2}.
\]
Take $c_0$ at most $g$ and the finite minimum of
$\delta_rs_r^2/2$ for $2\le r<16$. Increase $C_0$ to cover the finitely
many thresholds. Both constants are absolute and positive.
\end{proof}

\begin{proof}[Proof of Theorem~\ref{thm:main}]
Let $\alpha=\log(5/3)/3$, as in Lemma~\ref{lem:slack}, and choose
$0<c\le\min\{c_0,\alpha/C_0\}$. If $k\ge C_0rs_r^2$, the exponential
expression follows from Theorem~\ref{thm:large}. Otherwise
$ck/s_r^2\le\alpha r$, so \eqref{eq:slack} gives that expression.
The double-factorial expression holds in both ranges by
Lemma~\ref{lem:slack}. Taking their minimum proves the theorem.
\end{proof}

\begin{proof}[Proof of Corollary~\ref{cor:all-parameters}]
The multinomial bound holds for all $r,k\ge2$ by
\eqref{eq:off-diagonal-standard}. If $r,k\ge3$, use
Theorem~\ref{thm:main} and $(2/r!)M_r(k)\le r^{rk}$.
If $r=2$ and $k\ge C_0rs_r^2$, use Theorem~\ref{thm:large}.
For the remaining pairs with $r=2$, the elementary bound
$R_r(k)\le r^{r(k-1)}$ gives the desired exponential bound provided
$c\le\log2/C_0$. Finally, when $k=2$ and $r\ge3$, the same elementary
bound suffices after requiring $c\le1$, because
$2/s_r^2\le r\log r$. Reducing $c$ once to meet these finitely many
absolute restrictions proves the assertion.
\end{proof}
\medskip
\noindent\textbf{AI use statement.}
ChatGPT 6 was used to improve the language and presentation of this manuscript. The authors take full responsibility for its content.

\end{document}